\documentclass[10pt]{amsart}

\usepackage{comment, amsmath, amssymb, amsgen, amsthm, amscd, xspace, color, epsfig, float, epic}
\usepackage[hypertex]{hyperref}
\usepackage[all]{xy}

\makeatletter
\newcommand{\rmnum}[1]{\romannumeral #1}
\newcommand{\Rmnum}[1]{\expandafter\@slowromancap\romannumeral #1@}

\newcommand{\pt}{\partial}
\newcommand{\vf}{\varphi}

\newcommand{\Ext}{\operatorname{Ext}}

\newcommand{\Hom}{\operatorname{Hom}}

\newcommand{\htt}{\mathrm{ht}}

\newcommand{\Ker}{\mathrm{Ker}}

\newcommand{\Imm}{\mathrm{Im}}

\newcommand{\ra}{\rightarrow}

\newcommand{\pf}{\begin{proof}}
\newcommand{\epf}{\end{proof}}
\newcommand{\eq}{\begin{equation}}
\newcommand{\eeq}{\end{equation}}
\newcommand{\eqn}{\begin{equation*}}
\newcommand{\eeqn}{\end{equation*}}

\newcommand{\fra}{\mathfrak{a}}
\newcommand{\frb}{\mathfrak{b}}

\newcommand{\frg}{\mathfrak{g}}
\newcommand{\frh}{\mathfrak{h}}

\newcommand{\frl}{\mathfrak{l}}

\newcommand{\frn}{\mathfrak{n}}

\newcommand{\frp}{\mathfrak{p}}

\newcommand{\frs}{\mathfrak{s}}
\newcommand{\frt}{\mathfrak{t}}
\newcommand{\fru}{\mathfrak{u}}

\newcommand{\frsl}{\mathfrak{sl}}

\newcommand{\frso}{\mathfrak{so}}
\newcommand{\frsp}{\mathfrak{sp}}

\newcommand{\bbC}{\mathbb{C}}

\newcommand{\bbN}{\mathbb{N}}

\newcommand{\bbZ}{\mathbb{Z}}

\newcommand{\caF}{\mathcal{F}}

\newcommand{\caM}{\mathcal{M}}

\newcommand{\caO}{\mathcal{O}}

\newtheorem{theorem}[equation]{Theorem}
\newtheorem{cor}[equation]{Corollary}
\newtheorem{prop}[equation]{Proposition}
\newtheorem{lemma}[equation]{Lemma}

\theoremstyle{remark}
\newtheorem{remark}[equation]{Remark}

\theoremstyle{definition}

\newtheorem{example}[equation]{Example}

\numberwithin{equation}{section} 
\begin{document}

\title[Dirac cohomology and Euler-Poincar\'{e}]{Dirac cohomology and Euler-Poincar\'{e} pairing for weight modules}

\author{Jing-Song Huang}
\address[Huang]{Department of Mathematics, Hong Kong University of Science and technology,
Clear Water Bay, Kowloon, Hong Kong SAR, China}
\email{mahuang@ust.hk}

\author{Wei Xiao}
\address[Xiao]{College of Mathematics and statistics, Shenzhen University,
Shenzhen, 518060, Guangdong, China}
\email{xiaow@szu.edu.cn}

\thanks{This work is supported by NSFC Grant No. 11326059.}

\subjclass[2010]{17B10}

\maketitle

\begin{abstract}
Let $\frg$ be a reductive Lie algebra over $\bbC$. For any simple weight module of $\frg$ with finite-dimensional weight spaces, we show that its Dirac cohomology is vanished unless it is a highest weight module. This completes the calculation of Dirac cohomology for simple weight modules since the Dirac cohomology of simple highest weight modules was carried out in our previous work. We also show that the Dirac index pairing of two weight modules which have infinitesimal characters agrees with their Euler-Poincar\'{e} pairing. The analogue of this result for Harish-Chandra modules is a consequence of the Kazhdan's orthogonality conjecture which was settled by the first named author and Binyong Sun.
\end{abstract}
%
%
\section{Introduction}
%
%
In order to formally factorise an operator for Minkowski space, Dirac introduced a first-order differential operator which plays a significant role in the study of elementary particles. The Dirac operator and various analogues have remarkable impact in the development of physics and mathematics. The Dirac operator in mathematics was first introduced by Parthasarathy \cite{Par} to construct discrete series representations for real reductive Lie groups. Vogan \cite{V2} proposed an algebraic version of Dirac operator and defined the Dirac cohomology for symmetric pairs. It turns out the Dirac cohomology of Harish-Chandra modules determine their infinitesimal characters. This was conjectured by Vogan \cite{V2} and proved by the first named author and Pand\v{z}i\'{c} \cite{HP1}. The general case was defined and proved by Kostant \cite{Ko1, Ko3} in the setting of so called cubic Dirac operator. It was discovered gradually that the Dirac cohomology involves deeply with several classical objects in representation theory, including the discrete series and branching laws \cite{HP2, HPZ}.

The Dirac cohomology also has an intrinsic relation with Kostant's $\fru$-cohomology. The two invariants are isomorphic up to a twist for unitary representations of Hermitian types \cite{HPR}. We \cite{HX} obtained a similar isomorphism for simple highest weight modules in Kostant's setting and calculated explicitly the Dirac cohomology in this case. The highest weight modules are fundamental objects in the BGG category $\caO$ which was introduced by Bernstein-Gelfand-Gelfand \cite{BGG}. The category of all weight modules with finite weight multiplicities is a natural generalization of the category $\caO$. Our first motivation of this work is to calculate the Dirac cohomology of such weight modules. The study of these weight modules already led to far-reaching work of Benkart, Britten, Fernando, Futorny, Lemire, and others \cite{BBL, BL1, BL2, F, Fu}. It was proved by Fernando that every simple module with finite-dimensional weight spaces over a reductive Lie algebra $\frg$ is either cuspidal or parabolic induced from a simple cuspidal module over a reductive subalgebra of $\frg$. With this, we show that the Dirac cohomology of simple weight module is vanished unless it is a highest weight module. This completes the calculation of Dirac cohomology for simple weight modules.

Our second motivation of this work comes from the representation theory of reductive algebraic groups. This concerns two natural pairings between representations of a reductive algebraic group. The first one is the Euler-Poincar\'{e} pairing algebraically defined to be the alternating sum of the corresponding Extension groups. The second one is the elliptic pairing of analytic nature. It was conjectured by Kazhdan that these two pairings are equal. The case when the ground field is $p$-adic was proved independently by Schneider-Stuhler \cite{SS} and Bezrukavnikov \cite{Be}. The case of real reductive Lie group was proved by the first named author and Binyong Sun \cite{HS}. Let $G$ be a real reductive group in Harish-Chandra's class with corresponding complexified Lie algebra $\frg$, and let $K$ be a maximal compact subgroup of $G$. The Dirac index of every finite-length $(\frg, K)$-module is a virtual $K$-module associated with its Dirac cohomology. If $G$ is connected and $\mathrm{rank}G=\mathrm{rank}K$, the Dirac index pairing of two finite-length $(\frg, K)$-modules with infinitesimal characters agrees with their elliptic pairing. This was prove by the first named author \cite{H} and Renard \cite{R}. Renard also conjectures that the Dirac index pairing coincides with the Euler-Poincar\'{e} pairing in this case which is now a consequence of results in \cite{HS}.

Although the elliptic pairing is not applicable for weight modules, we can still define the Euler-Poincar\'{e} pairing and Dirac index pairing. It is a natural conjecture that there is a similar correspondence between these two pairings. It turns out that the proof of this conjecture relies heavily on a full classification of simple weight modules which was carried out by Mathieu \cite{Mat}. The so-called Mathieu's twisting functors plays a crucial role in this classification and in our proof. Our proof also depends on the results achieved in \cite{BKLM, CF, GS1, GS2} about the structure and properties of weight modules.

The paper is organized as follows. In section 2, we give the definition and basic properties of Dirac cohomology and the spin index. In section 3, we obtain the vinish theorem of Dirac cohomology for non highest weight modules. In section 4, we obtain equations of extension groups for weight modules by Mathieu's twisting functor. This result plays an essential role in the proof of our main theorem. In section 5, we give an inductive relation about EP pairings of parabolic induced modules and dual modules. In section 6, we prove our main result about EP pairing and Dirac index pairing.

%
%
\section{Dirac cohomology and the spin index}
%
%

Let $\frg$ be a reductive Lie algebra over $\bbC$ with Cartan subalgebra $\frh$. Let $\Delta\subset\frh^*$ be the root system of $(\frg, \frh)$, with corresponding root lattice $Q\subset\frh^*$. Let $\frg_\alpha$ be the space of root vectors associated with $\alpha\in\Delta$. We can choose $e_\alpha\in\frg_\alpha$ and $f_\alpha\in\frg_{-\alpha}$ such that $[e_\alpha, f_\alpha]=h_\alpha$, where $h_\alpha$ is the coroot corresponding to $\alpha\in\Delta$. Denote by $(,)$ the Killing form of $\frg$. Then $(e_\alpha, f_\alpha)=1$.

\subsection{Dirac cohomology} Now we give the setting of Dirac operators and Dirac cohomology in this paper, referring to \cite{HPR} for full details. Let $\frp$ be a parabolic subalgebra of $\frg$, with nilpotent radical $\fru$ and Levi subalgebra $\frl$ containing $\frh$. Denote by $\bar\fru$ the dual space of $\fru$ such that $\frg=\fru\oplus\frl\oplus\bar\fru$ and denote by $\frs$ the direct sum $\fru\oplus\bar\fru$. The Clifford algebra $C(\frs)$ of $\frs$ is defined by the relations
\begin{equation}\label{Crelation}
vw+wv=2(v,w),
\end{equation}
for any $v, w\in\frs$. Let $S_{\frg, \frl}$ be the spin representation of $C(\frs)$. Given a $\frg$-module $M$, there is a natural action of $U(\frg)\otimes C(\frs)$ on $M\otimes S_{\frg, \frl}$, where $U(\frg)$ is the universal enveloping algebra of $\frg$. It is natural to identify $\bar\fru$ with $\fru^*$ via $(,)$. Choose a basis $u_1, u_2, \ldots, u_n$ of $\fru$. Let $u_1^*, \ldots, u_n^*$ be the dual basis of $\bar\fru$ such that $(u_i, u_j^*)=\delta_{ij}$. Let
\[
C:=\sum_{i=1}^nu_i^*\otimes u_i-\frac{1}{4}\sum_{1\leq i,j\leq n}1\otimes u_iu_j[u_i^*,
u_j^*]
\]
and
\[
C^-:=\sum_{i=1}^nu_i\otimes u_i^*-\frac{1}{4}\sum_{1\leq i,j\leq n}1\otimes
u_i^*u_j^*[u_i, u_j]
\]
be two elements in $U(\frg)\otimes C(\frs)$. It was showed in \cite{HPR} that $C$ and $C^-$ are $\frl$-invariant and independent of the chosen $u_i$. Furthermore,
\[
C^2=0\quad\mbox{and}\quad(C^-)^2=0.
\]
By the invariance of $(,)$ and the identification $\frso(\frs)\simeq\bigwedge^2\frs$, there exists an embedding $\frl\ra\frso(\frs)\ra C(\frs)$, which gives $S_{\frg, \frl}$ a structure of $\frl$-module. Kostant \cite{Ko2} has pointed out that $S_{\frg, \frl}\simeq\bigwedge\fru\otimes\mathbb{C}_{\rho(\bar\fru)}$ as $\frl$-modules, where $\rho(\bar\fru)$ is the half sum of roots for $\bar\fru$ and $\mathbb{C}_{\rho(\bar\fru)}$ is a one-dimensional $\frl$-module with weight $\rho(\bar\fru)$. On the other hand, we have the following identifications:
\begin{equation}\label{coid1}
\textstyle\Hom(\bigwedge^p\bar\fru, M)\simeq\Hom((\bigwedge^p\fru)^*, M)\simeq M\otimes \bigwedge^p\fru.
\end{equation}
With these preliminaries out of the way, the differential operator $C$ on $M\otimes S_{\frg, \frl}$ corresponds to the coboundary operator $d$ on $\Hom(\bigwedge^p\bar\fru, M)$ which defines the $\bar\fru$-cohomology of $M$. In particular, there is a natural $\frl$-module isomorphism
\begin{equation}\label{coid2}
H(C, M\otimes S_{\frg, \frl})\simeq H^*(\bar\fru, M)\otimes\mathbb{C}_{\rho(\bar\fru)}.
\end{equation}
Similarly, the differential operator $C^-$ on $M\otimes S_{\frg, \frl}$ corresponds to $-2$ times the differential operator $\pt$ on $M\otimes\bigwedge^p\fru$ which defines the $\fru$-homology of $M$ and thus gives rise to a natural $\frl$-module isomorphism
\begin{equation}\label{coid22}
H(C^-, M\otimes S_{\frg, \frl})\simeq H_*(\fru, M)\otimes\mathbb{C}_{\rho(\bar\fru)}.
\end{equation}

The Dirac operator associated with the pair $(\frg,\frl)$ is defined by $D(\frg,\frl)=C+C^-$. Consider the action of $D(\frg,\frl)\in U(\frg)\otimes C(\frs)$ on $M\otimes S_{\frg, \frl}$:
\[
D(\frg, \frl):M\otimes S_{\frg, \frl}\ra M\otimes S_{\frg, \frl}.
\]
We call the $\frl$-module
\[
H_{D(\frg,\frl)}(M):=\Ker D(\frg, \frl)/\Ker D(\frg, \frl)\cap \Imm D(\frg, \frl)
\]
the Dirac cohomology of $M$ (with respect to $D(\frg, \frl)$). Since $D(\frg,\frl)$ is an odd operator, the $\bbZ_2$-grading of $S_{\frg, \frl}=S^+_{\frg, \frl}\oplus S^-_{\frg, \frl}$ gives a $\bbZ_2$-grading
\[
H_{D(\frg,\frl)}(M)=H^+_{D(\frg,\frl)}(M)\oplus H^-_{D(\frg,\frl)}(M).
\]

In general, let $\fra$ be a finite dimensional Lie algebra over $\bbC$, and let $\frt$ be a subalgebra of $\fra$. For a $\frt$-module $M$, we say $v\in M$ is \emph{$\frt$-finite} if $U(\frt)v$ is a finite-dimensional subspace of $M$. The $\frt$-module $M$ is \emph{locally $\frt$-finite} if every element in $M$ is $\frt$-finite. The $\fra$-module $M$ is called an $(\fra, \frt)$-module if $M$ is locally $\frt$-finite and semisimple when restrict to $\frt$. In fact, we only need to require $U(\frt)v$ to be finite-dimensional and semisimple for any $v\in M$ (e.g., \cite{L}, Chapter \Rmnum{17}, \S 2).  We say an $(\fra, \frt)$-module $M$ is \emph{admissible} if the multiplicity of each irreducible $\frt$-module in $M$ is finite. For convenience, we denote by $\mathcal{M}(\fra, \frt)$ the category of finitely generated admissible $(\fra, \frt)$-modules.

In particular, any $(\frg,\frh)$-module is called a \emph{weight module}. It is evident that any $(\frg, \frl)$-module is also a $(\frg, \frh)$-module and hence a weight module. If $M$ is an $\frh$-module and $\mu\in\frh^*$, the weight space of weight $\mu$ is
\[
M_\mu:=\{v\in M\ |\ hv=\mu(h)v, h\in\frh\}.
\]
The set of weights $\mu\in\frh^*$ such that $M_\mu\neq0$ is called the \emph{support} of $M$, denoted $\mathrm{supp} M$. It was stated in \cite{HX} (Proposition 4.8) that there exists an injective $\frl$-module homomorphism between Dirac chomology (relative to $D(\frg, \frl)$) and $\bar\fru$-cohomology (resp. $\fru$-homology) for modules with infinitesimal characters in the generalized BGG category $\caO^\frp$. This result can be easily generalized to the case of admissible $(\frg,\frl)$-modules by the same kind of reasoning.

\begin{prop}\label{inj1}
Let $M$ be an admissible $(\frg,\frl)$-module with infinitesimal character. Then there are injective $\frl$-module homomorphisms
\begin{equation*}
H_{D(\frg,\frl)}(M)\rightarrow H^*(\bar\fru, M)\otimes \mathbb{C}_{\rho(\bar\fru)}\quad\mbox{and}\quad H_{D(\frg,\frl)}(M)\rightarrow H_*(\fru, M)\otimes \mathbb{C}_{\rho(\bar\fru)}.
\end{equation*}
\end{prop}

Since $H_{D(\frg,\frl)}(M)$ does not depends on the choices of $\fru$, the equations holds for any $\fru$ such that $\frl\oplus\fru$ is a parabolic subalgebra. So we also have

\begin{prop}\label{inj2}
Let $M$ be an admissible $(\frg,\frl)$-module with infinitesimal character. Then there are injective $\frl$-module homomorphisms
\begin{equation*}
H_{D(\frg,\frl)}(M)\rightarrow H^*(\fru, M)\otimes \mathbb{C}_{\rho(\fru)}\quad\mbox{and}\quad H_{D(\frg,\frl)}(M)\rightarrow H_*(\bar\fru, M)\otimes \mathbb{C}_{\rho(\fru)},
\end{equation*}
where $\rho(\fru)$ is the half sum of roots for $\fru$.
\end{prop}

\subsection{Spin index} In this paper, a virtual $\frl$-module is an element in the Grothendieck group of the category of admissible $(\frl, \frl)$-modules. The $(\frg, \frl)$-\emph{spin index} of an admissible $(\frg, \frl)$-module $M$ is by definition the virtual $\frl$-module
\begin{equation}\label{index1}
I_{\frg,\frl}(M):=M\otimes S^+_{\frg, \frl}-M\otimes S^-_{\frg, \frl}.
\end{equation}

\begin{prop}[\cite{PS}, Proposition 1.5]\label{index2}
Let $M$ be an admissible $(\frg,\frl)$-module with infinitesimal character. Then
\begin{equation}\label{index21}
I_{\frg,\frl}(M)=H^+_{D(\frg,\frl)}(M)-H^-_{D(\frg,\frl)}(M)
\end{equation}
as virtual $\frl$-modules.
\end{prop}

Usually, the right side of (\ref{index21}) is called the $(\frg, \frl)$-Dirac index of $M$. If we let $M$ be a virtual $\frh$-module in the definition of $I_{\frg,\frl}(M)$ in (\ref{index1}), then $I_{\frg,\frl}(M)$ is also a virtual $\frh$-module. We have the following lemma.

\begin{lemma}\label{index3}
Let $M$ be an admissible $(\frg,\frh)$-module. Then
\[
I_{\frg,\frl}(M)=\sum_{i\in\bbN}(-1)^iH^i(\bar\fru,M)\otimes\bbC_{\rho(\bar\fru)}=\sum_{i\in\bbN}(-1)^iH_i(\fru,M)\otimes\bbC_{\rho(\bar\fru)}.
\]
as virtual $\frh$-modules, where $\bbN$ is the set of nonnegative integers.
\end{lemma}
\begin{proof}
The equations follow from the identification $S_{\frg, \frl}\simeq\bigwedge\fru\otimes\mathbb{C}_{\rho(\bar\fru)}$ and the Euler characteristic.
\end{proof}

The spin index with respect to $(\frg, \frl)$ is independent of the nilpotent radical $\fru$. So we also have

\begin{lemma}\label{index32}
Let $M$ be an admissible $(\frg,\frh)$-module. Then
\[
I_{\frg,\frl}(M)=\sum_{i\in\bbN}(-1)^iH^i(\fru,M)\otimes\bbC_{\rho(\fru)}=\sum_{i\in\bbN}(-1)^iH_i(\bar\fru,M)\otimes\bbC_{\rho(\fru)}.
\]
as virtual $\frh$-modules.
\end{lemma}

Since $S_{\frg, \frl}\otimes S_{\frl, \frh}\simeq S_{\frg, \frh}$ as $\frh$-modules, one obtains the following result.

\begin{lemma}\label{index4}
Let $M$ be an admissible $(\frg,\frh)$-module. Then
\begin{equation*}
I_{\frg,\frh}(M)=I_{\frl,\frh}(I_{\frg,\frl}(M))
\end{equation*}
as virtual $\frh$-modules.
\end{lemma}

%
%
\section{Dirac cohomology of weight modules}
%
%

In this section, we consider injective modules. At last, a vanished theorem of Dirac cohomology for weight modules is obtained.
\subsection{Injective envelop of $\Gamma$-bijective modules} Let $\fra$ be an $\frh$-invariant subalgebra of $\frg$ and $\frt$ a subalgebra of $\frh\cap\fra$. Denote by $\Delta(\fra)$ the set of roots for $\fra$ and by $Q(\fra)$ the root lattice of $\Delta(\fra)$. Following Mathieu's idea, we say $\Gamma\subseteq-\Delta(\fra)$ is a \emph{commuting set} if $\alpha+\beta\not\in\Delta$ for any $\alpha, \beta\in\Gamma$. Then $f_\alpha\in\fra$ for $\alpha\in\Gamma$. Let $F_\Gamma$ be the multiplicative subset of $U(\fra)$ generated by $(f_\alpha)_{\alpha\in\Gamma}$. Denote by $U_\Gamma$ be the localization of $U(\fra)$ relative to $F_\Gamma$. For any $\fra$-module $M$, denote by $M_\Gamma$ the localization $U_\Gamma\otimes_{U(\fra)}M$ of $M$. As in \cite{GS1}, we say that $M$ is \emph{$\Gamma$-injective} (resp. \emph{$\Gamma$-bijective}) if the action of $f_\alpha$ on $M$ is injective (resp. bijective) for all $\alpha\in\Gamma$. The following lemma is straightforward.

\begin{lemma}\label{bij}
If an $(\fra,\frt)$-module $M$ is $\Gamma$-injective, then $M\subseteq M_\Gamma$. If $M$ is  $\Gamma$-bijective, then $M\simeq M_\Gamma$ as $(\fra,\frt)$-modules.
\end{lemma}

Recall that an injective homomorphism $\vf: A\rightarrowtail B$ of $(\fra, \frt)$-modules is called \emph{essential} if $H\cap\vf(A)\neq0$ for any nonzero submodule $H$ of $B$. Moreover, we say that $B$ is an \emph{essential extension} of $\vf(A)$. Note that if $M$ is an essential extension of $N$ and $N$ is an essential extension of $L$, then $M$ is also an essential extension of $L$.

\begin{prop}\label{iee1}
Let $M$ be an $(\fra, \frt)$-modules. There exists a unique injective $(\fra, \frt)$-module $I$ containing $M$ such that $I$ is an essential extension of $M$.
\end{prop}

The injective module $I$ is called the \emph{injective envelop} of $M$. The proof of the above proposition uses standard methods of homological algebra, see for example Maclane \cite{M}, Chapter \Rmnum{3}, \S 11.

\begin{proof}
It was shown in \cite{BW}, Chapter \Rmnum{1}, \S2.6 that the category of $(\fra, \frt)$-modules has enough injective modules. Let $J$ be an injective $(\fra, \frt)$-module containing $M$. Consider the set $\Psi$ of essential extensions of $M$ contained in $J$. If $\{E_i\}$ is a subset of $\Psi$ ordered by inclusion, then the union $E=\cup E_i$ is still an $(\fra, \frt)$-module and an essential extension of $M$ contained in $J$. By Zorn's Lemma, there exists a maximal essential extension $I$ of $M$ which is contained in $J$. By a similar argument, we can find a maximal submodule $K$ of $J$ such that $K\cap I=0$. Therefore we have an essential injective homomorphism $\sigma:I\ra J/K$. Denote by $i$ the embedding $I\rightarrowtail J$. Since $J$ is an injective module, there is a homomorphism $\pi:J/K\ra J$ such that $i=\pi\sigma$. Keeping in mind that $i$ is injective, we have $\Ker\pi\cap\sigma(I)=0$. Hence $\Ker\pi=0$, for $\sigma$ is essential. It follows that $\pi(J/E)$ is an essential extension of $M$ contained in $J$. The maximality of $I$ implies $I\simeq J/K$. So we have a split exact sequence $0\ra K\ra J\ra I\ra0$. Therefore $I$ is a direct summand of $J$ and is injective.

Now the inclusion $e:M\ra I$ is essential with $I$ injective. If $e':M\ra I'$ is another such inclusion with $I'$ injective, then there exists a homomorphism $\pi: I\ra I'$ such that $\pi e=e'$. Since $e$ is essential and $e'$ is injective, a similar reasoning shows that $\pi$ is an injective map. Therefore $\pi(I)$ is an injective module and is a direct summand of $I'$. Because $e'=\pi e:M\ra I\ra I'$ is also essential, we must have $I\simeq I'$.
\end{proof}

\begin{lemma}\label{iee2}
Suppose that the injective homomorphism $\vf:M\ra N$ of $(\fra, \frt)$-modules is essential. If $M$ is $\Gamma$-injective, so is $N$.
\end{lemma}

\begin{proof}
If $M$ is $\Gamma$-injective, then for any $\alpha\in\Gamma$, $f_\alpha$ acts injectively on $M$. If $N$ is not $\Gamma$-injective, assume that $f_\alpha v=0$ for some nonzero vector $v\in N$. Since $N$ is an essential extension of $M$, the submodule $U(\fra)v$ generated by $v$ has a nontrivial intersection with $\vf(M)$, that is, there exists $u\in U(\fra)$ such that $u v\in \vf(M)$ and $uv\neq0$. On the other hand, we can find $u'\in U(\fra)$ such that $f_\alpha^ku=u'f_\alpha$ for positive integer $k$ large enough. Therefore, one has
\[
f_\alpha^k(uv)=(f_\alpha^ku)v=u'(f_\alpha v)=0.
\]
Since $uv$ is contained in $\vf(M)\simeq M$ on which the action of $f_\alpha$ is injective, one has $uv=0$, a contradiction. Thus $N$ is also $\Gamma$-injective.
\end{proof}

\begin{lemma}\label{ienv}
If an $(\fra, \frt)$-module $M$ is $\Gamma$-injective, then its injective envelop $I$ is $\Gamma$-bijective. If $M$ is $\Gamma$-bijective, so is $I/M$.
\end{lemma}

\begin{proof}
It suffices to prove the first statement since the second one is an easy consequence of the first one. In view of Proposition \ref{iee1} and Lemma \ref{iee2}, $I$ is $\Gamma$-injective. Then $I\subset I_\Gamma$ by Lemma \ref{bij}. Denote by $e$ the embedding $I\ra I_\Gamma$ and by $i$ the identity map $I\ra I$. Since $I$ is injective, there exists a map $\pi: I_\Gamma\ra I$ such that $\pi e=i$.

\[
\xymatrix{
  I \ar[d]_{i} \ar[r]^{e} &  I_\Gamma \ar@{.>}[ld]_{\pi}     \\
  I }
\]
Hence the exact sequence
\[
0\ra I\stackrel{e}{\longrightarrow} I_\Gamma\ra I_\Gamma/I\ra 0
\]
is split. This can happen only when $I\simeq I_\Gamma$.

\end{proof}

\begin{prop}\label{van}
Given an $\fra$-module $M$, if there exists $\alpha\in\Delta(\fra)$ such that $f_{-\alpha}$ acts bijectively on $M$, then $H^i(\fra, M)=0$ for any $i\in\bbN$.
\end{prop}
\begin{proof}
Let $\frt=0$ and $\Gamma=\{-\alpha\}$. Then $M$ is a $\Gamma$-bijective $(\fra, \frt)$-module. It follows from  Lemma \ref{ienv} that there exists an injective resolution of $M$ of $\Gamma$-bijective $\fra$-modules:
\[
0\ra M\ra I_0\ra I_1\ra\ldots.
\]
By definition $H^i(\fra, M)=\Ext_{\fra}^i(\bbC, M)$ is the cohomology group of the complex
\[
0\ra I_0^\fra\ra I_1^\fra\ra\ldots.
\]
But $I_i^\fra=0$ since $f_{-\alpha}\in\fra$ acts bijectively on $I_i$. Therefore $H^i(\fra, M)\equiv0$.
\end{proof}

\subsection{Dirac cohomology of weight modules} An admissible $(\frg, \frh)$-module $M$ is called \emph{cuspidal} if $e_\alpha$ acts injectively on $M$ for all $\alpha\in\Delta$. The following lemma is straightforward.

\begin{lemma}\label{cusp1}
Let $M$ be a simple admissible cuspidal $(\frg, \frh)$-module. There exists integer $d$ such that $\dim M_\lambda=d$ for any $\lambda\in\mathrm{supp}M$. Moreover, $\mathrm{supp}M=\lambda+Q$.
\end{lemma}

\begin{example}
Suppose that $e, f, h$ are the usual basis of $\frg=\frsl(2, \bbC)$. Then $\frs=\bbC e\oplus\bbC f$ and $\frh=\bbC h$. As in \cite{GS2}, denote by $t^\mu$ the monomial $t_0^{\mu_0}t_1^{\mu_1}$, where $\mu=(\mu_0, \mu_1)\in\bbC^2$. Set $\pt_0:=\pt/\pt t_0$ and $\pt_1:=\pt/\pt t_1$. Put
\[
\caF_\mu:=\left\{g\in t^\mu\bbC[t_0^{\pm1}, t_1^{\pm1}]\ \left|\ \left(t_0\pt_0+t_1\pt_1\right)g=(\mu_0+\mu_1)g\right.\right\}.
\]
The action of $\frg$ on $\caF_\mu$ is given by
\[
e\mapsto t_0\pt_1,\quad f\mapsto t_1\pt_0\quad\mbox{and}\quad h\mapsto t_0\pt_0-t_1\pt_1.
\]
Then $\caF_\mu$ is a cuspidal $\frg$-module if and only if $\mu_0, \mu_1\not\in\bbZ$.

Note that the spin module $S_{\frg, \frh}\simeq\bbC1\oplus\bbC e$ as vector spaces \cite{Ko2}. Consider the action of the Clifford algebra $C(\frs)$ on $\bbC1\oplus\bbC e$ given by
\[
e\cdot(1)=e,\quad e\cdot(e)=0, \quad f\cdot(1)=0, \quad\mbox{and}\quad f\cdot(e)=2.
\]
It is easy to verify that the above action is compatible with the relation (\ref{Crelation}). Since $D(\frg, \frh)=e\otimes f+f\otimes e$ in this case, one has
\[
D(\frg, \frh)(g_1\otimes 1+g_2\otimes e)=t_0\pt_1 g_2\otimes 2+t_1\pt_0 g_1\otimes e,
\]
where $g_1, g_2\in\caF_\mu$. Therefore
\[
\Imm D(\frg, \frh)=\left\{t_0\pt_1 g_2\otimes 2+t_1\pt_0 g_1\otimes e\ \left|\ g_1, g_2\in\caF_\mu  \right.\right\}
\]
and
\[
\Ker D(\frg, \frh)=\{g_1\otimes 1+g_2\otimes e\ |\ \pt_0g_1=0\ \mbox{and}\ \pt_1 g_2=0\ \mbox{for}\ g_1, g_2\in\caF_\mu \}.
\]
It is evident that $\Ker D(\frg, \frh)=0$ if $\mu_0, \mu_1\not\in\bbZ$. Hence
\[
H_{D(\frg,\frh)}(\caF_\mu)=0
\]
if $\caF_\mu$ is cuspidal.
\end{example}

Notice that any $(\frl, \frh)$-module $V$ can be made into a $(\frp, \frh)$-module by letting $\fru$ act trivially on $V$. Set
\[
M_\frp(V)=U(\frg)\otimes_{U(\frp)}V.
\]
If $V\in\caM(\frl, \frh)$ is simple, then $M_\frp(V)\in\caM(\frg, \frh)$ and admits a unique simple quotient which we denote by $L_\frp(V)$.

\begin{theorem}[Fernando]\label{fernando}
Every module in $\mathcal{M}(\frg, \frh)$ is of finite length. Every simple module $M\in\mathcal{M}(\frg, \frh)$ is isomorphic to $L_\frp(V)$ for some parabolic subalgebra $\frp$ containing $\frh$ and some simple cuspidal $(\frl, \frh)$-module $V$, where $\frl$ is the Levi subalgebra of $\frp$. If $\frg$ is simple and $M$ is cuspidal, then $\frg$ is either of type $A$ or of type $C$.
\end{theorem}

As a consequence of Proposition \ref{van} and Theorem \ref{fernando}, we have the following result about Dirac cohomology of weight modules.

\begin{theorem}\label{main1}
If $M\in\caM(\frg, \frh)$ is simple, then $$H_{D(\frg,\frh)}(M)=0$$ unless $M$ is a highest weight module.
\end{theorem}
\begin{proof}
In view of Theorem \ref{fernando}, one has $M\simeq L_\frp(V)$ for some parabolic subalgebra $\frp$. It follows from Corollary 2.4 in \cite{CF} that $L_\frp(V)$ is cuspidal as an $(\frl,\frh)$-module. Let $\frb$ be a Borel subalgebra contained in $\frp$. Denote by $\frn$ the maximal nilpotent radical of $\frb$ ($\frb=\frh\oplus\frn$). If $M$ is not a highest weight module, then the Levi subalgebra $\frl\neq\frh$ and $\Delta(\frl\cap\frn)\neq\emptyset$. For any $\alpha\in\Delta(\frl\cap\frn)$, the vector $f_{-\alpha}$ acts bijectively on $L_\frp(V)$ since it is a cuspidal $(\frl, \frh)$-module. Then $H^*(\frn, L_\frp(V))=0$ by Proposition \ref{van}. Now the statement follows from Proposition \ref{inj2}.
\end{proof}

\begin{remark}
We can also show that $H_{D(\frg,\frh)}(M_\frp(V))=0$ by a similar argument. The Dirac cohomology of highest weight modules was discussed in \cite{HX}. So the above theorem completes the calculation of Dirac cohomology for simple weight modules.
\end{remark}

%
%
\section{Mathieu's twisting functor and higher Ext groups}
%
%

In this section, we give equations of higher Ext groups for $\Gamma$-bijective weight modules twisted by an automorphism.

\subsection{Category of $\Gamma$-bijective modules}
Here we follow the notational conventions in the previous section but focus on the case when $\fra=\frg$ and $\frt=\frh$. From now on $\Gamma$ is a commuting set in $\Delta$, while $U_\Gamma$ is the localization of $U(\frg)$ relative to $F_\Gamma$ and $M_\Gamma$ is the localization $U_\Gamma\otimes_{U(\frg)}M$. Some preliminary results in this subsection was also stated in \cite{CM}.

\begin{lemma}\label{lemtst}
Let $M$ and $N$ be two $(\frg, \frh)$-modules. Consider the linear map
\[
\Gamma_{M,N}^*:\Hom_{\frg, \frh}(M, N)\ra\Hom_{U_\Gamma, \frh}(M_\Gamma, N_\Gamma)
\]
defined by
\[
\Gamma_{M,N}^*(\vf)(u\otimes v)=u\otimes\vf(v)
\]
for any $u\in U_\Gamma$, $v\in M$ and $\vf\in\Hom_{\frg, \frh}(M, N)$. Then
\begin{itemize}
\item [(\rmnum{1})] The map $\Gamma_{M,N}^*$ is well defined.
\item [(\rmnum{2})] The map $\Gamma_{M,N}^*$ sent injective maps to injective maps.
\item [(\rmnum{3})] If $M$ and $N$ are $\Gamma$-bijective, then $\Gamma_{M,N}^*$ is bijective.
\end{itemize}
\end{lemma}

\begin{proof}
Here we only consider the case when $\Gamma=\{\alpha\}$ for some $\alpha\in\Delta$, while the general case is similar. Note that
\[
1\otimes v=f_\alpha^{-1}\otimes f_\alpha v=\ldots=f_\alpha^{-m}\otimes f_\alpha^m v=...
\]
for any $v\in M$ and $m\in\bbN$. It is easy to see that $1\otimes v=0$ if and only if $f_\alpha^mv=0$ for some nonnegative integer $m$. Since $\Gamma=\{\alpha\}$, any $w\in M_\Gamma$ can be written as
\[
w=\sum_{i=0}^n{f_\alpha^{-i}u_i}\otimes v_i=f_\alpha^{-n}\left(1\otimes\sum_{i=0}^n{f_\alpha^{n-i}u_i} v_i\right)
\]
where $u_i\in U(\frg)$, $v_i\in M$ and $n\in\bbN$. Set $v=\sum_{i=0}^nf_\alpha^{n-i}u_i v_i$. Then $w=f_\alpha^{-n}(1\otimes v)$.

For the first statement, it suffices to prove that $\Gamma_{M,N}^*(\vf)(w)=0$ if $w=0$. In fact, we have $f_\alpha^{n}w=1\otimes v=0$ for $w=0$. Then there exists $m\in\bbN$ such that $f_\alpha^mv=0$. Therefore
\[
\Gamma_{M,N}^*(\vf)(w)=\sum_{i=0}^n{f_\alpha^{-i}u_i}\otimes \vf(v_i)=f_\alpha^{-n}\left(1\otimes \vf(v)\right)=f_\alpha^{-n-m}\left(1\otimes \vf(f_\alpha^{m}v)\right)=0.
\]
For the second statement, we need to show that $\Gamma_{M,N}^*(\vf)(w)=0$ yields $w=0$ for any injective $(\frg, \frh)$-homomorphism $\vf$. Indeed, if $\Gamma_{M,N}^*(\vf)(w)=0$, then
\[
0=f_\alpha^{n}\left(\Gamma_{M,N}^*(\vf)(w)\right)=f_\alpha^{n}\left(\sum_{i=0}^n{f_\alpha^{-i}u_i}\otimes \vf(v_i)\right)=1\otimes \vf(v).
\]
Thus $f_\alpha^{m}\vf(v)=\vf(f_\alpha^mv)=0$ for some $m\in\bbN$. We obtain $f_\alpha^mv=0$ since $\vf$ is injective. Therefore
\[
w=f_\alpha^{-n}(1\otimes v)=f_\alpha^{-n-m}\otimes f_\alpha^{m}v=0.
\]
The third statement follows from Lemma \ref{bij}.
\end{proof}

Let $\mathcal{C}_\Gamma$ be the category of $\Gamma$-bijective $(\frg, \frh)$-modules. Then $\mathcal{C}_\Gamma$ is abelian. Of course any $M\in\mathcal{C}_\Gamma$ may be regarded as a $(U_\Gamma, \frh)$-module since $M\simeq M_\Gamma$ as vector spaces and any $(\frg, \frh)$-homomorphism between modules in $\mathcal{C}_\Gamma$ can be naturally extended to a $(U_\Gamma, \frh)$-homomorphism. So $\mathcal{C}_\Gamma$ may also be viewed as a category of $(U_\Gamma, \frh)$-modules.

\begin{lemma}\label{lemtst2}
Given $I\in\mathcal{C}_\Gamma$, then $I$ is an injective $(\frg, \frh)$-module if and only if $I_\Gamma\simeq I$ is an injective $(U_\Gamma, \frh)$-module in $\mathcal{C}_\Gamma$. As a consequence, we have
\[
\Ext^i_{\frg, \frh}(M, N)\simeq\Ext^i_{\mathcal{C}_\Gamma}(M, N),
\]
for $M, N\in\mathcal{C}_\Gamma$ and $i\in\bbN$.
\end{lemma}
\begin{proof}
If $I$ is an injective $(\frg, \frh)$-module, then the injectivity of $I$ as a $(U_\Gamma, \frh)$-module is an immediate consequence of Lemma \ref{lemtst}(\rmnum{3}).

Now assume that $I\simeq I_\Gamma$ is an injective $(U_\Gamma, \frh)$-module in $\mathcal{C}_\Gamma$. For $(\frg, \frh)$-modules $M$ and $N$, suppose that we already have $(\frg, \frh)$-homomorphisms $i:M\ra I$ and $e: M\ra N$. Moreover, the map $e$ is injective. We need to find a $(\frg, \frh)$-homomorphism $\pi: N\ra I$ such that $\pi e=i$. With Lemma \ref{lemtst} in hand, denote $i_\Gamma:=\Gamma_{M, I}^*(i)$ and $e_\Gamma:=\Gamma_{M, N}^*(e)$. Then $e_\Gamma$ is injective by Lemma \ref{lemtst}(\rmnum{2}). Since $I$ is injective in $\mathcal{C}_\Gamma$, there exists a $(U_\Gamma, \frh)$-homomorphism $\pi_\Gamma:N_\Gamma\ra I_\Gamma$ such that $\pi_\Gamma e_\Gamma=i_\Gamma$.
\[
\xymatrix{
  M_\Gamma \ar[d]_{i_\Gamma} \ar[r]^{e_\Gamma} &  N_\Gamma \ar@{.>}[ld]_{\pi_\Gamma}     \\
  I_\Gamma }
\]
Since $I\simeq I_\Gamma$, define $\pi: N\ra I$ by $\pi(w):=\pi_\Gamma(1\otimes w)$ for any $w\in N$. Then $\pi$ is a $(\frg, \frh)$-homomorphism. By Lemma \ref{lemtst}, we get
\begin{equation*}
\begin{aligned}
\pi (e(v))=&\pi_\Gamma(1\otimes e(v))=\pi_\Gamma(e_\Gamma(1\otimes v))=(\pi_\Gamma e_\Gamma)(1\otimes v)\\
=&i_\Gamma(1\otimes v)=1\otimes i(v)\simeq i(v)
\end{aligned}
\end{equation*}
for $v\in M$. Therefore $I$ is an injective $(\frg, \frh)$-module.

For the second statement, it follows from Lemma \ref{ienv} that any $N\in\mathcal{C}_\Gamma$ has an injective resolution of $\Gamma$-bijective $(\frg, \frh)$-modules. It can also be viewed as an injective resolution in $C_\Gamma$ by the first statement. Then the equation of higher Ext groups follows from Lemma \ref{lemtst}(\rmnum{3}).
\end{proof}

\subsection{Mathieu's twisting functor} In \cite{Mat}, Mathieu's twisting functor plays a central role in the classification of irreducible admissible $(\frg, \frh)$-modules. Here we briefly describe it as follows. Suppose that the commuting set $\Gamma=\{\gamma_1, \gamma_2, \ldots, \gamma_l\}\subset\Delta$. There is a unique family of automorphisms $\Theta_{(x_1, \ldots, x_l)}:U_\Gamma\ra U_\Gamma$ with
\[
\Theta_{(x_1, \ldots, x_l)}(u)=f_{\gamma_1}^{x_1}\ldots f_{\gamma_l}^{x_l} uf_{\gamma_l}^{-x_l}\ldots f_{\gamma_1}^{-x_1},
\]
where all $x_i$'s are integers and $u\in U_\Gamma$. The map $(x_1, \ldots, x_l)\ra\Theta_{(x_1, \ldots, x_l)}(u)$ is a polynomial in $x_1, \ldots, x_l$ for any $u\in U_\Gamma$. So one can naturally extend the class of automorphisms to all $(x_1, \ldots, x_l)\in\bbC^l$. For $\nu=x_1\gamma_1+\ldots+x_l\gamma_l\in\frh^*$ and $(U_\Gamma, \frh)$-modules $M$, denote by $\Phi_\Gamma^\nu M$ the $(U_\Gamma, \frh)$-module $M$ twisted by the automorphism $\Theta_{(x_1, \ldots, x_l)}$. For $v\in M$, denote by $v^\nu$ the corresponding element in $\Phi_\Gamma^\nu M$. Then the action of $U_\Gamma$ on $\Phi_\Gamma^\nu M$ is given by
\[
u\cdot v^\nu=(\Theta_{(x_1, \ldots, x_l)}(u)v)^\nu,
\]
where $u\in U_\Gamma$ and $v\in M$. Then $\Phi_\Gamma^\nu$ is a functor on $\mathcal{C}_\Gamma$. Some standard facts about $\Phi_\Gamma^\nu$ include:

\begin{lemma}\label{lemtst3}
$\Phi_\Gamma^\nu$ is an exact and invertible functor on $\mathcal{C}_\Gamma$ which is viewed as a category of $(U_\Gamma, \frh)$-modules. In particular, $\Phi_\Gamma^\nu\circ \Phi_\Gamma^{-\nu}=\mathrm{Id}$.
\end{lemma}

As a consequence, we have the following result.

\begin{lemma}\label{lemtst4}
For $\nu\in\frh^*$, one has
\[
\Ext^i_{\mathcal{C}_\Gamma}(M, N)\simeq\Ext^i_{\mathcal{C}_\Gamma}(\Phi_\Gamma^\nu M, \Phi_\Gamma^\nu N),
\]
where $i\in\bbN$ and $M, N\in\mathcal{C}_\Gamma$.
\end{lemma}

In view Lemma \ref{ienv}, if $M$ and $N$ are $\Gamma$-bijective $(\frg, \frh)$-modules, then $\Ext_{\frg, \frh}^*(M,N)$ can be calculated within the category $\mathcal{C}_\Gamma$ using injective resolution of $\Gamma$-bijective $(\frg, \frh)$-modules. Combining Lemma \ref{lemtst2} and Lemma \ref{lemtst4}, we immediately get the following proposition.

\begin{prop}\label{proptst2}
For $\nu\in\frh^*$, $i\in\bbN$ and $M, N\in\mathcal{C}_\Gamma$, one has
\[
\Ext^i_{\frg, \frh}(M, N)\simeq\Ext^i_{\frg, \frh}(\Phi_\Gamma^{\nu} M, \Phi_\Gamma^\nu N).
\]
\end{prop}

\begin{remark}
The case $i=1$ of the above proposition is evident and well-known. For example, it is used to prove the main theorem in \cite{BKLM} which states that any admissible cuspidal $\frsp(2n, \bbC)$-module is reducible for $n>1$.
\end{remark}

The \emph{degree} $d(M)$ of an admissible $(\frg, \frh)$-module $M$ is the maximal dimension of its weight spaces: it is an integer $\geq0$, or $+\infty$. Mathieu gave a classification of simple cuspidal modules in $\caM(\frg, \frh)$ based on the following result.

\begin{lemma}\label{cusp2}
If $\frg$ is simple and $M\in\caM(\frg, \frh)$ is a simple cuspidal module, then there exist $\Gamma\subset\Delta$ and $\nu\in\frh^*$ such that $\Phi_\Gamma^{\nu} M$ contains a simple highest weight module of degree $d(M)$.
\end{lemma}

\begin{proof}
The existence is given by Mathieu's classification of coherent families; see Lemma 4.5., Proposition 4.8. and Proposition 6.2. in \cite{Mat}.
\end{proof}

\begin{lemma}\label{cusp3}
If $\frg\neq\frh$ and $M\in\caM(\frg, \frh)$ is a simple cuspidal module, then there exist $\Gamma\subset\Delta$ and $\nu\in\frh^*$ such that $\Phi_\Gamma^{\nu} M$ contains no cuspidal subquotient.
\end{lemma}

\begin{proof}
If $\frg$ is simple, by Lemma \ref{cusp2}, we can find $\Gamma\subset\Delta$ and $\nu\in\frh^*$ such that $\Phi_\Gamma^{\nu} M$ contains a highest weight module $L$ of degree $d(M)$. If $N$ is a cuspidal subquotient of $\Phi_\Gamma^{\nu} M$, then $d(N)\leq d(\Phi_\Gamma^{\nu} M)-d(L)=d(M)-d(M)=0$ by Lemma \ref{cusp1}, a contradiction. If $\frg$ is not simple, let $\frg=\frg_1\oplus\frg_2$, where $\frg_i$ is the Lie subalgebra with Cartan subalgebra $\frh_i=\frg_i\cap\frh$ and root system $\Delta_i$ for $i=1, 2$. In particular, let $\frg_1$ to be simple. It follows from Lemma 4.5 in \cite{F} that there are irreducible cuspidal module $M_i\in\caM(\frg_i, \frh_i)$ such that $M\simeq M_1\otimes M_2$. Since $\frg_1$ is simple, there exist $\Gamma_1 \in \Delta_1$ and $\nu_1 \in\frh_1^*$ such that $\Phi_{\Gamma_1}^{\nu_1} M_1$ contains no cuspidal $(\frg_1, \frh_1)$-subquotient. It is evident that $\frh_1^*$ can be viewed as a subspace of $\frh^*$ and $\Delta_1=\Delta\cap\frh^*_1$. Setting $\Gamma=\Gamma_1$ and $\nu=\nu_1$, then $\Phi_\Gamma^{\nu} M\simeq \Phi_{\Gamma_1}^{\nu_1} M_1\otimes M_2$. If $N$ is a simple cuspidal subquotient of $\Phi_\Gamma^{\nu} M$, again by Lemma 4.5 in \cite{F} we have $N\simeq N_1\otimes N_2$, where $N_i$ is a cuspidal $(\frg_i, \frh_i)$-module for $i=1, 2$. But $N_1$ is a $(\frg_1, \frh_1)$-subquotient of $\Phi_{\Gamma_1}^{\nu_1} M_1$ which contains no cuspidal subquotient, a contradiction.
\end{proof}
%
%
\section{Parabolic induced modules and dual modules}
%
%

In this section, we give equations of higher Ext groups for parabolic induced weight modules and dual weight modules.

\subsection{Euler-Poincar\'{e} pairing for parabolic induced modules}
Let $M$ and $N$ be two finitely generated admissible $(\frg, \frl)$-module, that is, $M, N\in\caM(\frg, \frl)$. The Euler-Poincar\'{e} pairing of them is defined by
\begin{equation}\label{epp}
\mathrm{EP}_{\frg,\frl}(M, N):=\sum_{i\in\bbN}(-1)^i\dim\Ext_{\frg,\frl}^i(M, N).
\end{equation}
The pairing is well-defined because of the following lemma (see \cite{BW}, Proposition 2.8).

\begin{lemma}
The extension group $\Ext_{\frg,\frl}^*(M, N)$ is finite dimensional for finitely generated admissible $(\frg, \frl)$-modules $M$ and $N$.
\end{lemma}

First, we prove the following proposition.

\begin{prop}\label{pind}
Suppose that $M\in\caM(\frg, \frh)$ and $V\in\caM(\frl, \frh)$. Then
\[
\mathrm{EP}_{\frg, \frh}(M_\frp(V), M)=\sum_{i\in\bbN}(-1)^i\mathrm{EP}_{\frl, \frh}(V, H^i(\fru, M)).
\]
\end{prop}

It will show in the next section that $H^i(\fru, M)\in\caM(\frl, \frh)$ for $M\in\caM(\frg, \frh)$ and $i\in \bbN$. So the pairing $\mathrm{EP}_{\frl, \frh}(V, H^i(\fru, M))$ is well defined.

\begin{proof}
Recall that $V$ can be viewed as a $\frp$-module with trivial $\fru$-action. First, it is evident that $U(\frg)\otimes_{U(\frp)}-$ is an exact functor on the category of $(\frp, \frh)$-modules. It sends projective objects to projective objects. Therefore, the Frobenius reciprocity yields
\begin{equation}\label{ind1}
\Ext_{\frg, \frh}^{i}(M_\frp(V), M)\simeq\Ext_{\frp, \frh}^{i}(V, M).
\end{equation}
On the other hand, consider the functor $(\cdot)^\fru$ from the category of $(\frp, \frh)$-modules to the category of $(\frl, \frh)$-modules (or the category of $(\frp, \frh)$-modules with trivial $\fru$-action). It sends injective objects to injective objects since
\begin{equation}\label{ind2}
\Hom_{\frp, \frh}(A, B)\simeq\Hom_{\frp, \frh}(A, B^\fru)\simeq\Hom_{\frl, \frh}(A, B^\fru)
\end{equation}
for any $(\frp, \frh)$-module $A$ with trivial $\fru$-action and any $(\frp, \frh)$-module $B$. Thus the covariant functor $\Hom_{\frp, \frh}(V, \cdot)$ on the category of $(\frp, \frh)$-modules is isomorphic to the composition of the functor $(\cdot)^\fru$ and the functor $\Hom_{\frl, \frh}(V, \cdot)$. Thus there exists a spectral sequence (e.g., \cite{W}, Theorem 5.8.3)
\begin{equation}\label{ind3}
E_2^{i,j}=\Ext_{\frl, \frh}^i(V, H^j(\fru, M))\Rightarrow\Ext_{\frp, \frh}^{i+j}(V, M).
\end{equation}
Then the proposition follows from (\ref{ind1}) and (\ref{ind3}).
\end{proof}

\begin{cor}\label{cind}
Let $M\in\caM(\frg, \frh)$ be a cuspidal $(\frg, \frh)$-module and $V\in\caM(\frl, \frh)$. If $\frp\neq\frg$, then
\[
\Ext_{\frg, \frh}^{i}(M_\frp(V), M)=0,
\]
for any $i\in\bbN$.
\end{cor}
\begin{proof}
It is evident that $\frp\neq\frg$ implies $\fru\neq0$. In view of Proposition \ref{van}, one gets $H^j(\fru, M)=0$ for any $j\in\bbN$. Then the corollary is a consequence of (\ref{ind1}) and (\ref{ind3}).
\end{proof}

\subsection{Duality of weight modules} Let $\tau$ be a standard anti-involution of $\frg$ which interchanges $e_\alpha$ and $f_\alpha$ for $\alpha\in\Delta$ and fix $h\in\frh$ (see for example \cite{Hum}, \S0.5). For any $(\frg, \frh)$-module $M$, there exists a dual space $$M^\vee:=\bigoplus_{\mu\in\frh^*}(M_\mu)^*,$$ where $(M_\mu)^*$ can be naturally identified with the space of those $f\in M^*$ which are vanished on all weight spaces $M_\lambda$ unless $\lambda=\mu$.  The $\frg$-action on $M^\vee$ is defined by
\[
(X\cdot f)(v)=f(\tau(X)\cdot v)
\]
for $X\in\frg$ and $v\in M$.

\begin{lemma}\label{dual1}
Let $(\cdot)^\vee$ be the dual functor defined above. Then
\begin{itemize}
\item [(\rmnum{1})] The functor $(\cdot)^\vee$ is a contravariant exact functor on the category of $(\frg, \frh)$-modules.
\item [(\rmnum{2})] The functor $(\cdot)^\vee$ is invertible on the category of admissible $(\frg, \frh)$-modules. Moreover, one has $M^{\vee\vee}\simeq M$ if $M$ is an admissible $(\frg, \frh)$-module.
\item [(\rmnum{3})] If $M$ is a simple admissible $(\frg, \frh)$-module, then $M^\vee\simeq M$.

\item [(\rmnum{4})] If $M\in\mathcal{M}(\frg, \frh)$, then $I_{\frg, \frh}(M^\vee)=I_{\frg, \frh}(M)$.
\end{itemize}
\end{lemma}
\begin{proof}
The statements (\rmnum{1}) and (\rmnum{2}) are obvious. For (\rmnum{3}), by Theorem \ref{fernando}, there exist a parabolic subalgebra $\frp$ of $\frg$ and a simple cuspidal $(\frl,\frh)$-module $V$ such that $M\simeq L_\frp(V)$. Here $\frl$ which contains $\frh$ is the Levi subalgebra of $\frp$. Since $V_\mu=M_\mu$ for $\mu\in\mathrm{supp}V$, denote by $V^\vee$ the subspace $\bigoplus_{\mu\in\mathrm{supp}V}(M_\mu)^*$ of $L_\frp(V)^\vee$. Then $V^\vee$ is an $\frl$-submodule of $L_\frp(V)^\vee$. Note that for $f\in V^\vee$,
\[
(\fru f)(M)=f(\tau(\fru)M)=f(\bar\fru M)\subset f(\bigoplus_{\mu\not\in\mathrm{supp}V}M_\mu)=0.
\]
So $V^\vee$ is a $\frp$-submodule of $L_\frp(V)^\vee$ with trivial $\fru$-action. The Frobenius reciprocity implies
\[
\Hom_\frp(V^\vee, L_\frp(V)^\vee)\simeq \Hom_\frg(M_\frp(V^\vee), L_\frp(V)^\vee).
\]
The simple module $L_\frp(V)^\vee$ is a quotient of $M_\frp(V^\vee)$ and thus $L_\frp(V)^\vee\simeq L_\frp(V^\vee)$. Applying Theorem \ref{fernando} again, it suffices to prove (\rmnum{3}) for cuspidal modules of simple Lie algebra of type $A$ or $C$. This was shown in \cite{GS2} and \cite{BKLM}. Since $M^\vee$ and $M$ have the same weight spaces for $M$ admissible, (\rmnum{4}) is an immediate consequence.
\end{proof}

Except using projective and injective resolutions, the extension group $\Ext^i_{\frg, \frh}(M, N)$ can be defined by equivalent classes of long exact sequence of $(\frg, \frh)$-modules. Any exact sequence
\[
\textbf{E}: 0\ra N\ra E_i\ra\ldots\ra E_1\ra M\ra0
\]
of $(\frg, \frh)$-modules is called an \emph{$i$-extension} of $M$ by $N$. Two $i$-extensions $\textbf{E}$ and $\textbf{F}$ are equivalent if there exists a commutative diagram
\[
\xymatrix{
  \textbf{E}: 0 \ar[r]
                &  N\ar@{=}[d]\ar[r] & E_i\ar[r]\ar[d] & \ldots\ar[r] & E_1\ar[r]\ar[d] & M \ar@{=}[d] \ar[r] & 0\\
  \textbf{F}: 0 \ar[r]
                &  N\ar[r] & F_i\ar[r] & \ldots\ar[r] & F_1\ar[r] & M \ar[r] & 0.            }
\]

\begin{prop}\label{depair}
If $M, N\in\caM(\frg, \frh)$, then
\[
\Ext^i_{\frg, \frh}(M, N)\simeq\Ext^i_{\frg, \frh}(N^\vee, M^\vee).
\]
\end{prop}
\begin{proof}
Suppose that
\[
\xymatrix@C=0.5cm{
  \mathbf{E}: 0 \ar[r] & N \ar[rr]^{\theta_{i+1}} && E_i \ar[rr]^{\theta_{i}} && \ldots \ar[rr]^{\theta_2} && E_1 \ar[rr]^{\theta_{1}} && M \ar[r] & 0}
\]
represents an equivalent class $\xi\in\Ext^i_{\frg, \frh}(M, N)$. Then
\[
\xymatrix@C=0.5cm{
  \mathbf{E}^\vee: 0 \ar[r] & M^\vee \ar[rr]^{\theta_{1}^\vee} && E_1^\vee \ar[rr]^{\theta_{2}^\vee} && \ldots \ar[rr]^{\theta_i^\vee} && E_i^\vee \ar[rr]^{\theta_{i+1}^\vee} && N^\vee \ar[r] & 0}.
\]
represents an element $\xi^\vee$ in $\Ext^i_{\frg, \frh}(N^\vee, M^\vee)$. It is easy to verify that $\xi\ra\xi^\vee$ induces a linear map from $\Ext^i_{\frg, \frh}(M, N)$ to $\Ext^i_{\frg, \frh}(N^\vee, M^\vee)$. Similarly, there is linear map
\[
\Ext^i_{\frg, \frh}(N^\vee, M^\vee)\ra\Ext^i_{\frg, \frh}(M^{\vee\vee}, N^{\vee\vee})\simeq\Ext^i_{\frg, \frh}(M, N)
\]
which sends $\xi^\vee$ to $\xi^{\vee\vee}$ (the last isomorphism follows from Lemma \ref{dual1}). The $i$-extension $\xi^{\vee\vee}$ is represented by
\[
\xymatrix@C=0.5cm{
  \mathbf{E}^{\vee\vee}: 0 \ar[r] & N \ar[rr]^{\theta_{i+1}^{\vee\vee}} && E_i^{\vee\vee} \ar[rr]^{\theta_{i}^{\vee\vee}} && \ldots \ar[rr]^{\theta_2^{\vee\vee}} && E_1^{\vee\vee} \ar[rr]^{\theta_{1}^{\vee\vee}} && M \ar[r] & 0}.
\]
For any $(\frg, \frh)$-module $L$, let $\iota_L: L\ra L^{\vee\vee}$ be the linear map defined by $\iota_L(v)(\vf)=\vf(v)$ for any $v\in L$ and $\vf\in L^\vee$. It is a standard result of homological algebra that $\iota$ yields a natural transformation from the identity functor to the functor $(\cdot)^{\vee\vee}$ of the category of $(\frg, \frh)$-modules. Then $\mathbf{E}$ is equivalent to $\mathbf{E}^{\vee\vee}$. Hence $\xi^{\vee\vee}=\xi$ and $\xi\ra\xi^\vee$ is bijective.
\end{proof}

%
%
\section{Euler-Poincar\'{e} pairing for weight modules}
%
%

In this section, we will prove the following theorem.

\begin{theorem}\label{main2}
If $M, N\in\mathcal{M}(\frg, \frh)$, then
\begin{equation}\label{epp}
\mathrm{EP}_{\frg,\frh}(M, N)=[I_{\frg,\frh}(M), I_{\frg,\frh}(N)]_\frh.
\end{equation}
\end{theorem}

The pair on the right side is defined as follows. For virtual $\frh$-modules $A$, denote by $A^+$ and $A^-$ the positive and negative part of $A$, where $A^+$ and $A^-$ are actual $\frh$-modules such that $\Hom_\frh(A^+, A^-)=0$. Given virtual $\frh$-modules $A$ and $B$, define the pair
\begin{equation}
\begin{aligned}
\left[A, B\right]_\frh=&\dim\Hom_\frh(A^+, B^+)-\dim\Hom_\frh(A^+, B^-)\\
&-\dim\Hom_\frh(A^-, B^+)+\dim\Hom_\frh(A^-, B^-).
\end{aligned}
\end{equation}
 We say that $A$ is finite-dimensional if both $A^+$ and $A^-$ are finite-dimensional. The pair is of course well-defined if both $A$ and $B$ are finite-dimensional.

\subsection{The pairing of spin index} This subsection is devoted to show that the spin indices of weight modules involved here are finite dimensional. To prove it, one has to understand the structure of modules in $\mathcal{M}(\frg, \frh)$. Let $Z(\frg)$ be the center of $U(\frg)$. We say that a $(\frg, \frh)$-module $M$ is \emph{$Z(\frg)$-finite} if there exists a finite codimension ideal of $Z(\frg)$ which annihilates $M$. Denote by $Z(g)'$ the set of algebra homomorphisms $Z(\frg)\ra\bbC$. For any $Z(\frg)$-finite module $M$ and $\chi\in Z(\frg)'$, define the $\chi$-primary component of $M$ to be
\[
M^\chi:=\left\{v\in M\ |\ (z-\chi(z))^n v=0,\ \mbox{for some}\ n\in\bbN\ \mbox{and all}\ z\in Z(\frg)\right\}.
\]
We say that $M$ has generalized $Z(\frg)$-infinitesimal character $\chi$ if $M=M^\chi$. For $\nu\in\frh^*$, denote $\bar\nu:=\nu+Q\in\frh^*/Q$. Then
\[
M^\chi_{\bar\nu}:=\bigoplus_{\mu\in\bar\nu}M^\chi\cap M_\mu
\]
is a $(\frg, \frh)$-submodule of $M$. In view of Proposition 0.3.19 in \cite{V1} and Lemma 4.1 in \cite{F}, we have
\begin{equation}\label{ep0}
M=\bigoplus_{\chi\in Z(\frg)', \bar\nu\in\frh^*/Q}M^\chi_{\bar\nu}
\end{equation}
for any $Z(\frg)$-finite $(\frg, \frh)$-module $M$. In view of Theorem \ref{fernando}, every module in $\mathcal{M}(\frg, \frh)$ is of finite length and thus $Z(\frg)$-finite. For convenience, denote by $\mathcal{M}^\chi_{\bar\nu}$ the full subcategory of $\mathcal{M}(\frg, \frh)$ whose objects are modules $M$ for which $M={M}^\chi_{\bar\nu}$. It is evident that we can write
\[
\mathcal{M}(\frg, \frh)=\bigoplus_{\chi\in Z(\frg)', \bar\nu\in\frh^*/Q}\mathcal{M}^\chi_{\bar\nu}.
\]

Let $Z_\frg$ (resp. $Z_\frl$) be the center of $\frg$ (resp. $\frl$). Then $\frg=Z_\frg\oplus[\frg, \frg]$ and $$Z_\frg=\{h\in\frh\ |\ \alpha(h)=0\ \mbox{for all}\ \alpha\in\Delta\}\subset\frh\cap Z(\frg).$$
Given $\chi\in Z(\frg)'$ and $\nu\in\frh^*$, both of them can be viewed as functions on $Z_\frg$.

\begin{lemma}\label{ep1}
Let $M$ be a $Z(\frg)$-finite $(\frg, \frh)$-module. Suppose that $M^\chi_{\bar\nu}\neq0$ for some $\chi\in Z(\frg)'$ and $\bar\nu\in\frh^*/Q$.
\begin{itemize}
\item [(\rmnum{1})] We have $\chi|_{Z_\frg}=\nu|_{Z_\frg}$ as functions restricted to $Z_\frg$.

\item [(\rmnum{2})] For any $\chi^\frl\in Z(\frl)'$, there exists at most one set $\lambda+Q(\frl)$ such that $\lambda\in\bar\nu$ and $\chi^\frl|_{Z_\frl}=\lambda|_{Z_\frl}$ as functions restricted to $Z_\frl$.
\end{itemize}
\end{lemma}
\begin{proof}
Let $v$ be a nonzero weight vector in $M^\chi_{\bar\nu}$, with weight $\lambda\in\nu+Q$. Then for any $h\in Z_\frg\subseteq\frh$,
\[
hv=\lambda(h)v=\nu(h)v.
\]
Since $M^\chi_{\bar\nu}$ has generalized infinitesimal character $\chi$, there exists $n\in\bbN$ such that
\[
(h-\chi(h))^nv=0=(\nu(h)-\chi(h))^nv.
\]
Thus $\nu(h)=\chi(h)$ for $h\in Z_\frg$.

If there are $\lambda, \mu\in\bar\nu$ such that $\chi^\frl|_{Z_\frl}=\lambda|_{Z_\frl}=\mu|_{Z_\frl}$, then $\lambda-\mu\in Q$ and $(\lambda-\mu)(Z_\frl)=0$. By Lemma 2.7 in \cite{CF}, this can happen only when $\lambda-\mu\in Q(\frl)$.
\end{proof}

Let $W_\frg$ be the Weyl group of the pair $(\frg, \frh)$.

\begin{lemma}\label{ep2}
The category $\mathcal{M}^\chi_{\bar\nu}$ contains only finitely many nonisomorphic simple modules for any $\chi\in Z(\frg)'$ and $\bar\nu\in\frh^*/Q$. If $M$ is an admissible $Z(\frg)$-finite $(\frg, \frh)$-module, then $M^\chi_{\bar\nu}\in\mathcal{M}^\chi_{\bar\nu}$.
\end{lemma}
\begin{proof}
By Theorem \ref{fernando}, every simple module in $\mathcal{M}^\chi_{\bar\nu}$ has the form $L_\frp(V)$ for some parabolic subalgebra $\frp$ containing $\frh$ and some simple cuspidal $(\frl, \frh)$-module $V$, where $\frl$ is the Levi subalgebra of $\frp$. First, since every parabolic subalgebra $\frp$ containing $\frh$ is determined by a subset of $\Delta$ which is finite, there are only finitely many choices of $\frp$. Second, since $L_\frp(V)$ has infinitesimal character $\chi$, there exists at most $[W_\frg:W_\frl]$ characters $\chi^\frl\in Z(\frl)'$ corresponding to $V$ (see for example the remark before Lemma 2.7 in \cite{CF}). Third, we can find $\lambda\in\frh^*$ such that $\mathrm{supp}V=\lambda+Q(\frl)$ for the simple cuspidal $(\frl, \frh)$-module $V$ by Lemma \ref{cusp1}. Since $L_\frp(V)\in\mathcal{M}^\chi_{\bar\nu}$, one has $\lambda\in\mathrm{supp}V\subseteq\bar\nu$. Applying Lemma \ref{ep1} (\rmnum{1}) to $Z(\frl)$-finite $(\frl, \frh)$-module $V$, we have $\chi^\frl|_{Z_\frl}=\lambda|_{Z_\frl}$ as functions on $Z_\frl$. By Lemma \ref{ep1} (\rmnum{2}), there exists at most one such set $\lambda+Q(\frl)$. At last, by Mathieu's classification, there can be at most finitely many choices of cuspidal $(\frl, \frh)$-module $V$ when $\chi^\frl$ and $\mathrm{supp}V$ are fixed (see \cite{Mat}, \S8, \S9, also summarized in \cite{BKLM} and \cite{GS2}). Hence the number of nonisomorphic  simple modules $L_\frp(V)$ in $\mathcal{M}^\chi_{\bar\nu}$ is finite.

Let $N$ be any finitely generated submodule of $M^\chi_{\bar\nu}$. Since $M$ is admissible, then $N\in\caM^\chi_{\bar\nu}$ and the multiplicity of every composition factor in $N$ has a fixed upper bound (depending on $M$). Since $\mathcal{M}^\chi_{\bar\nu}$ contains only finitely many nonisomorphic simple modules, the length of $N$ also has a fixed upper bound (depending on $M$). This can happen only when $M^\chi_{\bar\nu}$ is finitely generated, that is, $M^\chi_{\bar\nu}\in\mathcal{M}^\chi_{\bar\nu}$.
\end{proof}

\begin{lemma}\label{ep3}
If $M\in \mathcal{M}(\frg,\frh)$, then $H^i(\fru, M)\in\mathcal{M}(\frl,\frh)$ for any $i\in\bbN$.
\end{lemma}
\begin{proof}
For $M\in\mathcal{M}(\frg,\frh)$, it is evident that $H^i(\fru, M)$ is an admissible $(\frl, \frh)$-module. It is enough to prove that $H^i(\fru, M)$ is of finite length as an $(\frl, \frh)$-module. Since $M$ is of finite length, we only need to consider the case when $M$ is simple, that is, $M$ has infinitesimal character and $\mathrm{supp}M\subseteq\nu+Q$ for some $\nu\in\frh^*$. The Casselman-Osborne Theorem \cite{CO} implies that $H^i(\fru, M)$ is $Z(\frl)$-finite. With (\ref{ep0}) in hand, one obtains
\[
H^i(\fru, M)=\bigoplus_{\chi\in Z(\frl)', \bar\lambda\in\frh^*/Q(\frl)}H^i(\fru, M)_{\bar\lambda}^{\chi}.
\]
In view of Corollary 3.1.6 in \cite{V1}, there exist at most finitely many different generalized $Z(\frl)$-infinitesimal character $\chi$ occurring in $H^i(\fru, M)$. Thus it suffices to show that
\[
H^i(\fru, M)^{\chi}=\bigoplus_{\bar\lambda\in\frh^*/Q(\frl)}H^i(\fru, M)_{\bar\lambda}^{\chi}
\]
is contained in $\mathcal{M}(\frl,\frh)$ when $\chi\in Z(\frl)'$ is fixed. Indeed, if $H^i(\fru, M)_{\bar\lambda}^{\chi}\neq0$, then one has $\lambda\in\mathrm{supp}H^i(\fru, M)+Q(\frl)\subseteq\nu+Q$. Moreover $\chi|_{Z_\frl}=\lambda|_{Z_\frl}$ as functions on $Z_\frl$ by Lemma \ref{ep1}(\rmnum{1}). It follows from Lemma \ref{ep1}(\rmnum{2}) that the set $\bar\lambda=\lambda+Q(\frl)$ is unique. Therefore
\[
H^i(\fru, M)^{\chi}=H^i(\fru, M)_{\bar\lambda}^{\chi}
\]
and is contained in $\mathcal{M}(\frl,\frh)$ by Lemma \ref{ep2} since $H^i(\fru, M)$ is admissible and $Z(\frl)$-finite.
\end{proof}

\begin{lemma}\label{ep4}
Let $M\in\mathcal{M}(\frg,\frh)$. The virtual module $I_{\frg,\frh}(M)$ is finite-dimensional. In general, the virtual module $I_{\frl,\frh}(H^i(\fru, M))$ is finite-dimensional for any $i\in\bbN$.
\end{lemma}
\begin{proof}
It suffices to show this for simple modules in $\mathcal{M}(\frg,\frh)$. It follows from Proposition \ref{index2} that
\[
I_{\frg,\frh}(M)=H^+_{D(\frg,\frh)}(M)-H^-_{D(\frg,\frh)}(M)
\]
for simple module $M\in\mathcal{M}(\frg,\frh)$. So the first assertion is a consequence of Theorem 6.16 in \cite{HX} and Theorem \ref{main1}, while the general result follows from Lemma \ref{ep3}.
\end{proof}

\subsection{Parabolic induced modules} In this subsection we prove several results about parabolic induced modules which will be used in the proof of Theorem \ref{main2}. Let $\frp\supset\frl\supset\frh$ be a parabolic subalgebra with Levisubalgebra $\frl$ containing $\frh$. Suppose that $\frb=\frh\oplus\frn$ is a Borel subalgebra contained in $\frp$ and $B$ is the unique simple system associated with the nilradical $\frn$. Then $\Delta(\frl)$ is generated by a subset $I\subset B$. Define the \emph{parabolic height} (or $I$-height) of $\nu\in\frh^*$ as
\[
\htt_\frp\nu=\sum_{\alpha\in B\backslash I}k_\alpha,
\]
where $\nu=\sum_{\alpha\in B}k_\alpha\alpha$. It is not difficult to see that $\htt_\frp$ does not depend on $\frn$. We say a simple weight module $M\in\caM(\frg, \frh)$ is \emph{$\frp$-induced} if $M\simeq L_\frp(V)$ for some simple cuspidal $(\frl, \frh)$-module $V$.

\begin{prop}\label{pind2}
Let $\frp$ be a parabolic subalgebra of $\frg$ with corresponding Levi subalgebra $\frl$ containing $\frh$. Let $V$ be a simple cuspidal $(\frl, \frh)$-module.
\begin{itemize}
\item [(\rmnum{1})] There exists $\chi\in Z(\frg)'$ and $\bar\nu\in\frh^*/Q$ such that $M_\frp(V)\in\mathcal{M}^\chi_{\bar\nu}$. Furthermore $M_\frp(V)$ has a Jordan-H\"{o}lder series with any irreducible subquotient isomorphic to $L_\frp(W)\in\mathcal{M}^\chi_{\bar\nu}$ for some simple cuspidal $(\frl, \frh)$-module $W$.

\item [(\rmnum{2})] There exists an $\bbN$-valued function $a_\frp$ on simple $\frp$-induced modules in $\mathcal{M}^\chi_{\bar\nu}$ such that $a_\frp(L_\frp(W))<a_\frp(L_\frp(V))$ if $L_\frp(W)$ is a subquotient of $M_\frp(V)$ and $W\not\simeq V$. In particular, if $a_\frp(L_\frp(V))=0$, then $L_\frp(V)=M_\frp(V)$.
\end{itemize}
\end{prop}
\begin{proof}
(\rmnum{1}) See for example Corollary 1.11 and Theorem 2.8 in \cite{CF}.

(\rmnum{2}) Let $L_\frp(V)$ be a simple module in $\mathcal{M}^\chi_{\bar\nu}$. For all $\lambda\in\mathrm{supp}V$, $\htt_\frp\lambda$ has a fixed value since $\htt_\frp(Q(\frl))=0$. We see that $\lambda\in\mathrm{supp}V\subset\nu+Q$ implies $\htt_\frp(\lambda-\nu)\in\bbZ$. Denote $a'_\frp(L_\frp(V)):=\htt_\frp(\lambda-\nu)$. In view of Lemma \ref{ep2}, $\mathcal{M}^\chi_{\bar\nu}$ contains only finitely many nonisomorphic simple modules. We can define
\[
a_\frp(L_\frp(V))=a'_\frp(L_\frp(V))-\min\{a'_\frp(L_\frp(W))\ |\ L_\frp(W)\in\mathcal{M}^\chi_{\bar\nu}\ \}.
\]
Then $a_\frp(L_\frp(V))\in\bbN$. For simplicity, denote
\[
Q^-(\fru):=\{\sum_{\alpha\in\Delta(\bar\fru)}k_\alpha\alpha\ |\ k_\alpha\in\bbN\}.
\]
Since $M_\frp(V)\simeq U(\bar\fru)\otimes V$ as $\frh$-modules, then $\mathrm{supp}M_\frp(V)=\mathrm{supp} V+Q^-(\fru)$. If $L_\frp(W)$ is a subquotient of $M_\frp(V)$, one has $\mathrm{supp} W\subseteq \mathrm{supp} V+Q^-(\fru)$, that is, there exists $\gamma\in Q^-(\fru)$ such that $a_\frp(L_\frp(W))=a_\frp(L_\frp(V))+\htt_\frp\gamma\leq a_\frp(L_\frp(V))$. If the equality holds, then $\gamma=0$ and $\mathrm{supp} W\subseteq \mathrm{supp} V$. Moreover, $W$ can be viewed as an $(\frl, \frh)$-submodule of $M_\frp(V)\simeq V\oplus N$, where
\[
N:=\bigoplus_{\gamma\not\in\mathrm{supp}V}M_\frp(V)_\gamma
\]
is an $(\frl, \frh)$-submodule of $M_\frp(V)$. We must have $W\subset V$ and thus $W=V$ as simple $(\frl, \frh)$-modules. If $a_\frp(L_\frp(V))=0$, then $M_\frp(V)$ has no subquotient other than $L_\frp(V)$ by the first statement. Hence $M_\frp(V)=L_\frp(V)$.
\end{proof}

\begin{lemma}\label{ep5}
For $V\in\caM(\frl, \frh)$, we have
\[
I_{\frg, \frl}(M_\frp(V))=V\otimes\bbC_{\rho(\fru)}.
\]
\end{lemma}
\begin{proof}
Since $M_\frp(V)\simeq U(\bar\fru)\otimes V$ is a free $U(\bar\fru)$-module, one obtains $H_0(\bar\fru, M_\frp(V))=V$ and $H_i(\bar\fru, M_\frp(V))=0$ for $i\neq0$. Then the statement follows from Lemma \ref{index32}.
\end{proof}

\begin{lemma}\label{ep6}
Let $\Gamma$ be a commuting set of $\Delta$ and $M$ be a $\Gamma$-bijective module in $\mathcal{M}(\frg, \frh)$. If $\Gamma$ is a basis of $Q$, then
\[
I_{\frg, \frl}(M)=0.
\]
\end{lemma}
\begin{proof}
In view of (\ref{ep0}),  we only need to consider the case when $M=M_{\bar\nu}^\chi$ for some $\chi\in Z(\frg)'$ and $\nu\in\frh^*$. Then $\mathrm{supp}M\subseteq\bar\nu= \nu+Q$. Since $\Gamma$ is a basis of $Q$ and $M$ is $\Gamma$-bijective, we get $\mathrm{supp} M=\nu+Q$ and $\dim M_\mu=\dim M_\nu$ for any $\mu\in\nu+Q$. Denote $n:=\dim M_\nu$. It suffices to show that
\[
\dim(M\otimes S^+_{\frg, \frl})_\lambda-\dim(M\otimes S^-_{\frg, \frl})_\lambda=0
\]
for any $\lambda\in\frh^*$. Recall that $S_{\frg, \frl}\simeq\bigwedge\fru\otimes\bbC_{\rho(\bar\fru)}$ as $\frl$-modules. If $\lambda\not\in\nu+\rho(\bar\fru)+Q$, then
\[
\dim(M\otimes S^+_{\frg, \frl})_\lambda=\dim(M\otimes S^-_{\frg, \frl})_\lambda=0.
\]
If $\lambda\in\nu+\rho(\bar\fru)+Q$, then
\[
\begin{aligned}
\dim(M\otimes S^+_{\frg, \frl})_\lambda-\dim(M\otimes S^-_{\frg, \frl})_\lambda
=&\sum_{\mu\in\frh^*}\dim M_{\lambda-\mu}\left(\dim(S^+_{\frg, \frl})_\mu-\dim(S^-_{\frg, \frl})_\mu\right)\\
=&n\left(\dim S^+_{\frg, \frl}-\dim S^-_{\frg, \frl}\right)\\
=&n\sum_{i\in\bbN}(-1)^i\dim\wedge^i\fru=n(1-1)^{\dim\fru}=0.
\end{aligned}
\]
\end{proof}

\subsection{Proof of Theorem \ref{main2}} It suffices to prove the theorem for simple modules $M, N\in\caM(\frg, \frh)$. We proceed by induction on the rank of $[\frg, \frg]$. If $\mathrm{rank}[\frg,\frg]=0$, that is, $\frg=\frh$, then $M$ and $N$ are 1-dimensional weight modules. Therefore
\[
\mathrm{EP}_{\frh,\frh}(M, N)=\dim\Hom_\frh(M, N)=[M, N]_\frh=[I_{\frh,\frh}(M), I_{\frh,\frh}(N)]_\frh.
\]
Now suppose that $\mathrm{rank}[\frg, \frg]>0$. If $M$ is not cuspidal, by Fernando's theorem, we can assume that $M=L_\frp(V)$ for some proper parabolic subalgebra $\frp$ of $\frg$ with Levi subalgebra $\frl$ containing $\frh$ and cuspidal $(\frl, \frh)$-module $V$. Since $\mathrm{rank}[\frl,\frl]<\mathrm{rank}[\frg, \frg]$, we can apply the induction hypothesis to Proposition \ref{pind} and get
\[
\mathrm{EP}_{\frg, \frh}(M_\frp(V), N)=\sum_{i\in\bbN}(-1)^i\left[I_{\frl, \frh}(V), I_{\frl, \frh}(H^i(\fru, N))\right]_\frh.
\]
Using Lemma \ref{ep5} and Lemma \ref{index32}, one obtains
\[
\mathrm{EP}_{\frg, \frh}(M_\frp(V), N)=\left[I_{\frl, \frh}(I_{\frg, \frl}(M_\frp(V))), I_{\frl, \frh}(I_{\frg, \frl}(N))\right]_\frh.
\]
It follows from Lemma \ref{index4} that
\[
\mathrm{EP}_{\frg, \frh}(M_\frp(V), N)=\left[I_{\frg, \frh}(M_\frp(V)), I_{\frg, \frh}(N)\right]_\frh.
\]
With Proposition \ref{pind2} in hand, we can get
\[
\mathrm{EP}_{\frg, \frh}(L_\frp(V), N)=\left[I_{\frg, \frh}(L_\frp(V)), I_{\frg, \frh}(N)\right]_\frh
\]
by induction on $a_\frp(L_\frp(V))$. If $M$ is cuspidal and $N$ is not cuspidal, in view of Lemma \ref{dual1} and Proposition \ref{depair}, we can prove the theorem by a similar argument.

If both $M$ and $N$ are simple cuspidal $(\frg, \frh)$-modules. By Lemma \ref{cusp3}, there exists a commuting set $\Gamma'\subset \Delta$ and $\nu\in\frh^*$ such that $\Phi_{\Gamma'}^\nu M$ has no cuspidal subquotient. The set $\Gamma'$ can be extended to a commuting set $\Gamma\subset \Delta$ such that $\Gamma$ is a basis of $Q$. Since $\Phi_{\Gamma}^\nu M=\Phi_{\Gamma'}^\nu M$ has no cuspidal subquotient, the argument in the previous paragraph yields
\[
\mathrm{EP}_{\frg, \frh}(\Phi_\Gamma^{\nu} M, \Phi_\Gamma^\nu N)=[I_{\frg, \frh}(\Phi_\Gamma^{\nu} M),I_{\frg, \frh}(\Phi_\Gamma^\nu N)].
\]
It is easy to see that both $M$ and $\Phi_\Gamma^\nu M$ are $\Gamma$-bijective. By Lemma \ref{ep6}, we have $I_{\frg, \frh}(M)=I_{\frg, \frh}(\Phi_\Gamma^\nu M)=0$. It follows from Proposition \ref{proptst2} that
\begin{equation}
\begin{aligned}
\mathrm{EP}_{\frg, \frh}(M, N)=&\mathrm{EP}_{\frg, \frh}(\Phi_\Gamma^{\nu} M, \Phi_\Gamma^\nu N)=[I_{\frg, \frh}(\Phi_\Gamma^{\nu} M),I_{\frg, \frh}(\Phi_\Gamma^\nu N)]\\
=&0=[I_{\frg, \frh}(M),I_{\frg, \frh}(N)].
\end{aligned}
\end{equation}

As a consequence of Proposition \ref{index2}, we have the following result.

\begin{theorem}\label{main22}
If $M, N\in\mathcal{M}(\frg, \frh)$ have infinitesimal characters, then
\begin{equation}\label{epp}
\mathrm{EP}_{\frg,\frh}(M, N)=[H^+_{D(\frg,\frh)}(M)-H^-_{D(\frg,\frh)}(M), H^+_{D(\frg,\frh)}(N)-H^-_{D(\frg,\frh)}(N)]_\frh.
\end{equation}
\end{theorem}

\begin{remark}
If the Dirac cohomology in the above theorem is replaced by the higher Dirac cohomology defined in \cite{PS}, then the equation (\ref{epp}) can be extended to all $M, N\in\mathcal{M}(\frg, \frh)$.
\end{remark}

Combined with Theorem \ref{main1}, we have the following corollary.

\begin{cor}
Let $M, N\in\caM(\frg, \frh)$. If $M$ is simple and not a highest weight module, then
\[
\mathrm{EP}_{\frg,\frh}(M, N)=\mathrm{EP}_{\frg,\frh}(N, M)=0.
\]
\end{cor}

\end{document}